\documentclass[12pt]{amsart}
\usepackage{amsfonts}
\usepackage{amsfonts,latexsym,rawfonts,amsmath,amssymb,amsthm}
\usepackage[plainpages=false]{hyperref}
\usepackage{graphicx}

\numberwithin{equation}{section}

\RequirePackage{color}
\theoremstyle{plain}

\newtheorem{theorem}{Theorem}[section]

\newtheorem{definition}{Definition}[section]

\newtheorem{lemma}[theorem]{Lemma}
\newtheorem{proposition}[theorem]{Proposition}
\newtheorem{remark}[theorem]{Remark}

\newcommand{\beq}{\begin{equation}}
\newcommand{\eeq}{\end{equation}}
\newcommand{\beqs}{\begin{eqnarray*}}
\newcommand{\eeqs}{\end{eqnarray*}}
\newcommand{\beqn}{\begin{eqnarray}}
\newcommand{\eeqn}{\end{eqnarray}}
\newcommand{\beqa}{\begin{array}}
\newcommand{\eeqa}{\end{array}}

\def\phi{\varphi}

\begin{document}
\title[Prescribed shifted Gauss curvatures]{Horo-convex hypersurfaces with prescribed shifted Gauss curvatures in $\mathbb{H}^{n+1}$}

\author{Li Chen}
\address{Faculty of Mathematics and Statistics, Hubei Key Laboratory of Applied Mathematics, Hubei University,  Wuhan 430062, P.R. China}
\email{chenli@hubu.edu.cn}

\author{Kang Xiao}
\address{Faculty of Mathematics and Statistics, Hubei Key Laboratory of Applied Mathematics, Hubei University,  Wuhan 430062, P.R. China}
\email{dalangbeikang@163.com}

\author{Qiang Tu$^\ast$}
\address{Faculty of Mathematics and Statistics, Hubei Key Laboratory of Applied Mathematics, Hubei University,  Wuhan 430062, P.R. China}
\email{qiangtu@hubu.edu.cn}



\thanks{This research was supported by funds from Hubei Provincial Department of Education
Key Projects D20181003.}
\thanks{$\ast$ Corresponding author}

\begin{abstract}
In this paper, we consider prescribed shifted Gauss curvature
equations for horo-convex hypersurfaces in $\mathbb{H}^{n+1}$. Under
some sufficient condition, we obtain an existence result by the
standard degree theory based on the a prior estimates for the
solutions to the equations. Different from the prescribed Weingarten
curvature problem in space forms, we do not impose a sign condition
for radial derivative of the functions in the right-hand side of the
equations to prove the existence due to the horo-covexity of
hypersurfaces in $\mathbb{H}^{n+1}$.
\end{abstract}

\maketitle

\baselineskip18pt

\parskip3pt

 \section{Introduction}

Different from hypersurfaces in $\mathbb{R}^{n+1}$, there are four
different kinds of convexity for hypersurfaces in $\mathbb{H}^{n+1}$
\cite{Ale90, Ale93}. One of them is the horo-convexity which is
defined as follows.
\begin{definition}
A smooth hypersurface $M \subset \mathbb{H}^{n+1}$ is called
horo-convex if $\kappa_i(p)>1$ for all $p \in M$ and $1\leq i\leq
n$, where $\kappa=(\kappa_1, . . . , \kappa_n)$ are the principal
curvatures of $M\subset \mathbb{H}^{n+1}$ which are defined by
eigenvalues of the Weingarten matrix $\mathcal{W}=(h^{j}_{i})$.
\end{definition}
Geometrically, a hypersurface $M \subset \mathbb{H}^{n+1}$ is called
horo-convex if and only if it is convex by horospheres, where the
horospheres in hyperbolic space are hypersurfaces with constant
principal curvatures equal to $1$ everywhere. In \cite{Esp09}, the
authors suggest that horospheres can be naturally regarded in many
ways as hyperplanes in the hyperbolic space $\mathbb{H}^{n+1}$. This
fact implies the similarity between the horo-convixty in
$\mathbb{H}^{n+1}$ and the convexity in $\mathbb{R}^{n+1}$ from
geometric aspect .

Furthermore, this interesting formal similarities, or the general
similarities between the geomerty of horo-convex regions in
$\mathbb{H}^{n+1}$ (that is, regions which are given by the
intersection of a collection of horo-balls), and that of convex
Euclidean bodies have been deeply explored in \cite{Esp09} and
\cite{And20}. In particular, Andrews-Chen-Wei \cite{And20} introduce
the shifted Weingarten matrix
\begin{eqnarray*}
\mathcal{W}-I
\end{eqnarray*}
for the hypersurface $M \subset \mathbb{H}^{n+1}$ along the line
that the convexity of hypersurfaces in $\mathbb{R}^{n+1}$ can be
describes by the positive definite of their Weingarten matrix.
Clearly, $M \subset \mathbb{H}^{n+1}$ is horo-convex if and only if
the shifted Weingarten matrix $\mathcal{W}-I$ is positive definite.
Moreover, they \cite{And20} define the shifted principal curvatures
by
\begin{eqnarray*}
(\lambda_1, . . . , \lambda_n):=(\kappa_1 -1, . . ., \kappa_n -1),
\end{eqnarray*}
which are eigenvalues of the shifted Weingarten matrix
$\mathcal{W}-I$. Thus, similar to $k$-mean curvature, we can define
$k$-th shifted mean curvature for $1\leq k\leq n$ by
\begin{eqnarray*}
\sigma_k(\kappa-1):=\sigma_k(\kappa_1-1, ...,
\kappa_n-1)=\sum_{i_1<i_2< \cdot\cdot\cdot<
i_k}(\kappa_{i_1}-1)\cdot\cdot\cdot(\kappa_{i_k}-1).
\end{eqnarray*}
Since
\begin{eqnarray}\label{sum}
\sigma_k(\kappa-1)=\sum_{i=0}^{k}(-1)^{k-i}C^{n-k}_{n-i}\sigma_{i}(\kappa),
\end{eqnarray}
the $k$-th shifted mean curvature can also be regarded as the linear
combination of the $i$-mean curvature $\sigma_{i}(\kappa)$ of $M$
for $1\leq i\leq k$, where
$C^{n-k}_{n-i}=\frac{(n-i)!}{(n-k)!(k-i)!}$.

In \cite{And20, Hu20}, curvature flows for horo-convex hypersurfaces
in hyperbolic space with speed given by the function $f$ of the
shifted principal curvatures $\lambda_i$ have been extensively
studied by Andrews-Chen-Wei and Hu-Li-Wei respectively. As
applications, they prove some new geometric inequalities involving
the weighted integral of $k$-th shifted mean curvature for
horo-convex hypersurfaces. Later, Wang-Wei-Zhou \cite{Wang20} study
inverse shifted curvature flow in hyperbolic space.

The shifted curvatures can also be well be interpreted in the work
\cite{Esp09} by Espinar-G\'alvez-Mirain  on the extension of the
Christoffel problem \cite{Chr65, Fi} to space forms. They show that
the Christoffel problem can be naturally formulated in the context
of hypersurfaces $M^n\subset \mathbb{H}^{n+1}$ by introducing the
hyperbolic curvature radii of
 $M^n\subset \mathbb{H}^{n+1}$, defined as
\begin{eqnarray*}
\mathcal{R}_i:=\frac{1}{|\kappa_i-1|}.
\end{eqnarray*}
Clearly, the hyperbolic curvature radii of the horo-convex
hypersurface $M^n\subset \mathbb{H}^{n+1}$ are in fact the shifted
principle curvature radii. Thus, the shifted Gauss curvature can be
interpreted as
\begin{eqnarray*}
\sigma_n(\lambda_1, ...,
\lambda_n)=\frac{1}{\mathcal{R}_1\cdot\cdot\cdot \mathcal{R}_n}
\end{eqnarray*}
in a very simple way to the Gauss curvature of hypersurfaces in
Euclidean.

All the work above motivate us to study furtherly the geometry and
analysis of horo-convex hypersurfaces with the shifted curvatures in
$\mathbb{H}^{n+1}$. In this paper, we consider the the problem of
prescribed shifted curvatures for horo-convex hypersurfaces in
$\mathbb{H}^{n+1}$.

Let $M$ be a horo-convex hypersurface in $\mathbb{H}^{n+1}$. We can
choose a point $o$ inside $M$ such that $M$ is star-shaped with
respect to $o$, thus $M$ can be parametrized as a radial graph over
$\mathbb{S}^n$. So we consider geodesic polar coordinates centered
at $o$, the hyperbolic space $\mathbb{H}^{n+1}$ can be regarded as a
warped product space $[0, +\infty)\times \mathbb{S}^n$ with metric
\begin{eqnarray*}
\overline{g}=dr^2+sinh^2r g_{\mathbb{S}^n},
\end{eqnarray*}
where $ g_{\mathbb{S}^n}$ is the standard sphere metric. Thus, $M$
can be represented by
\begin{eqnarray*}
M=\{(x, r(x)): x \in \mathbb{S}^n\}.
\end{eqnarray*}
Let $h_{ij}$ be the second fundamental form of $M$ and $g$ be the
induced metric of $M$. Denote by $\widetilde{h}_{ij}=h_{ij}-g_{ij}$.
Thus
\begin{eqnarray*}
\det(\widetilde{h}_{ij}(X))=(\kappa_1-1)(\kappa_2-1)\cdot\cdot\cdot(\kappa_n-1)
\end{eqnarray*}
is the shifted Gauss curvature of $M$, $\kappa(X)=(\kappa_1(X), ...,
\kappa_n(X))$ are the principle curvatures of hypersurface $M$ at
$X$. In this paper, we study the problem of prescribed shifted Gauss
curvature
\begin{eqnarray}\label{Eq}
\det(\widetilde{h}_{ij}(X))=f(x, r(x)),
\end{eqnarray}
on a horo-convex hypersurface $M\subset \mathbb{H}^{n+1}$, where
$X=(x, r(x)) \in M$ and $f$ is given smooth functions in
$\mathbb{S}^{n}\times [0, +\infty)$. Clearly, the problem \eqref{Eq}
can also be regarded as the problem of the prescribed linear
combination of the Weingarten curvatures in view of \eqref{sum}.

We mainly get the following theorem.

\begin{theorem}\label{Main}
Let $n\ge3$ and $f(x, r)\in C^\infty(\mathbb{S}^{n}\times [0,
+\infty))$ be a positive function, assume that
\begin{eqnarray}\label{ASS1}
\coth r-1 \geq f(x, r)\quad \mbox{for} \quad r\geq r_2,
\end{eqnarray}
\begin{eqnarray}\label{ASS2}
\coth r-1 \leq f(x, r) \quad \mbox{for} \quad r\leq r_1.
\end{eqnarray}
Then there exists at least a smooth horo-convex, closed hypersurface
$M$ in $\{(x, r)\in \mathbb{H}^{n+1}: r_1\leq r\leq r_2, x \in
\mathbb{S}^n\}$ satisfies equation (\ref{Eq}).
\end{theorem}

\begin{remark}
For the prescribed Weingarten curvatures problem, it is need usually
to impose a sign condition for radial derivative of $f$ in order to
derive a prior gradient estimate. However, we do not need such
condition in Theorem \ref{Main} due to the horo-covexity of the
hypersurface in $\mathbb{H}^{n+1}$.
\end{remark}

The prescribed Weingarten curvature equation
$\sigma_{k}(\kappa(X))=f(X)$ has been widely studied in the past two
decades. Such results were obtained for case of prescribed mean
curvature by Bakelman-Kantor \cite{Ba1, Ba2} and Treibergs-Wei
\cite{Tr}. For the case of prescribed Gaussian curvature by Oliker
\cite{Ol}. For general Weingarten curvatures by Aleksandrov
\cite{Al}, Firey \cite{Fi}, Caffarelli-Nirenberg-Spruck \cite{Ca}
for a general class of fully nonlinear operators $F$, including
$F=\sigma_k$ and $F=\frac{\sigma_k}{\sigma_l}$. Some results have
been obtained by Li-Oliker \cite{Li-Ol} on unit sphere, Barbosa-de
Lira-Oliker \cite{Ba-Li} on space forms, Jin-Li \cite{Jin} on
hyperbolic space, Andrade-Barbosa-de Lira \cite{An} on warped
product manifolds, Li-Sheng \cite{Li-Sh} for Riemannain manifold
equipped with a global normal Gaussian coordinate system.

For prescribed curvature problems in the case $f$ also depends on
the normal vector field $\nu$ along the hypersurface $M$, see
Caffarelli-Nirenberg-Spruck \cite{Ca}, Ivochkina \cite{Iv1, Iv2},
Guan-Li-Li \cite{Guan09}, Guan-Lin-Ma \cite{Guan12}, Guan-Guan
\cite{Guan02}, Guan-Ren-Wang \cite{Guan-Ren15}, Li-Ren-Wang
\cite{Li-Ren}, Chen-Li-Wang \cite{Chen} \cite{Ren} and Ren-Wang
\cite{Ren, Ren1}.

The organization of the paper is as follows.
In Sect. 2
we start with some preliminaries.
$C^0$, $C^1$ and $C^2$ estimates are given in Sect. 3.
In Sect. 4 we prove theorem \ref{Main}.

\section{Preliminaries}

\subsection{Setting and General facts}
For later convenience, we first state our conventions on Riemann
Curvature tensor and derivative notation. Let $M$ be a smooth
manifold and $g$ be a Riemannian metric on $M$ with Levi-Civita
connection $\nabla$. For a $(s, r)$-tensor field $\alpha$ on $M$,
its covariant derivative $\nabla \alpha$ is a $(s, r+1)$-tensor
field given by
\begin{eqnarray*}
&&\nabla \alpha(Y^1, .., Y^s, X_1, ..., X_r, X)
\\&=&\nabla_{X} \alpha(Y^1, .., Y^s, X_1, ..., X_r)\\&=&X(\alpha(Y^1, .., Y^s, X_1, ..., X_r))-
\alpha(\nabla_X Y^1, .., Y^s, X_1, ..., X_r)\\&&-...-\alpha(Y^1, ..,
Y^s, X_1, ..., \nabla_X  X_r).
\end{eqnarray*}
The coordinate expression of which is denoted by
$$\nabla \alpha=(\alpha_{k_{1}\cdot\cdot\cdot
k_{r}; k_{r+1}}^{l_{1}\cdot\cdot\cdot l_{s}}).$$ We can continue to
define the second covariant derivative of $\alpha$ as follows:
\begin{eqnarray*}
&&\nabla^2 \alpha(Y^1, .., Y^s, X_1, ..., X_r, X, Y)
=(\nabla_{Y}(\nabla\alpha))(Y^1, .., Y^s, X_1, ..., X_r, X).
\end{eqnarray*}
The coordinate expression of which is denoted by
$$\nabla^2 \alpha=(\alpha_{k_{1}\cdot\cdot\cdot
k_{r}; k_{r+1}k_{r+2}}^{l_{1}\cdot\cdot\cdot l_{s}}).$$ Similarly,
we can also define the higher order covariant derivative of
$\alpha$:
$$\nabla^3 \alpha=\nabla(\nabla^2 \alpha), \nabla^4 \alpha=\nabla(\nabla^3 \alpha), ... ,$$
and so on. For simplicity, the coordinate expression of the
covariant differentiation will usually be denoted by indices without
semicolons, e.g.,
$$u_{i}, \quad u_{ij} \quad \mbox{or} \quad
u_{ijk}$$ for a function $u: M\rightarrow \mathbb{R}$.

Let $X: M\rightarrow \mathbb{H}^{n+1}$ be an immersed hypersurface
with the standard metric $\overline{g}$ and Levi-Civita connection
$\overline{\nabla}$. Then£¬ $M$ can get the induced metric $g$ and
the Levil-Civita connection $\nabla$ of $g$. Pick a local coordinate
chart $\{x^i\}_{i=1}^{n}$ on $M$, then $e_i=\frac{\partial
X}{\partial x^i}$ form a local frame field of $X(M)$. Let $\nu$ be a
given unit normal and $h_{ij}$ be the second fundamental form $A$ of
the hypersurface with respect to $\nu$, that is
$$h_{ij}=-\langle\overline{\nabla}_{e_i}
\overline{\nabla}_{e_j}X, \nu\rangle_{\overline{g}}.$$

Recalling the Codazzi equation
\begin{equation}\label{Codazzi}
\nabla_{k}h_{ij}=\nabla_{j}h_{ik}
\end{equation}
and Simons' identity (see also \cite{And20})
\begin{eqnarray}\label{2rd}
\nabla_{(i}\nabla_{j)}h_{kl}
&=&\nabla_{(k}\nabla_{l)}h_{ij}+h_{ij}h_{km}h^{m}_{l}-h_{kl}h_{i}^{m}h_{mj}-g_{ij}h_{kl}+g_{kl}h_{ij},
\end{eqnarray}
where the brackets denote symmetrisation.

\subsection{Star-shaped hypersurfaces in $\mathbb{H}^{n+1}$}

Let $M$ be a closed hypersurface containing the origin in
$\mathbb{H}^{n+1}=[0, +\infty)\times \mathbb{S}^n$ with metric
\begin{eqnarray*}
\overline{g}=dr^2+sinh^2r g_{\mathbb{S}^n},
\end{eqnarray*}
where $ g_{\mathbb{S}^n}$ is the standard sphere metric. Then, we
have the following lemma (See \cite{Hu20}).

\begin{lemma}
Assume $M$ can be parametrized as a radial graph over $\mathbb{S}^n$
\begin{eqnarray*}
M=\{(x, r(x)): x \in \mathbb{S}^n\}.
\end{eqnarray*}
Let $\{x^1, ..., x^n\}$ be a local coordinate on $\mathbb{S}^n$,
$\{\partial_1, ..., \partial_n\}$ be the corresponding tangent
vector filed and $\partial_r$ be the radial vector field in
$\mathbb{R}^{n+1}$. $D_i\varphi=D_{\partial_i} \varphi$,
$D_iD_j\varphi=D^2 \varphi(\partial_i, \partial_j)$ denote the
covariant derivatives of $\varphi$ with respect to the round metric
$\sigma$ of $\mathbb{S}^n$. Then, the tangential vector takes the
form
$$e_i=\partial_i+r_i \partial_r.$$
The induced metric on $M$ has
$$g_{ij}=D_ir D_jr+\sinh^2r\sigma_{ij}$$
We also have the outward unit normal vector of $\Sigma$
$$\nu=\frac{1}{v}\bigg(\partial_r-\sinh^{-2} rD^{j}r\partial_{j}\bigg),$$
where $(\sigma^{ij})=(\sigma_{ij})^{-1}$ and $D^ir=\sigma^{ij}D_jr$.
Define a new function $u: \mathbb{S}^{n}\rightarrow\mathbb{R}$ by
\begin{equation}\label{02.01}
u(\theta)=\int_{c}^{r(\theta)}\frac{1}{\sinh s}ds.
\end{equation}
Then the induced metric on $M$ has
$$g_{ij}=\sinh^2r(D_iu D_ju+\sigma_{ij})$$
with the inverse
$$g^{ij}=\sinh^{-2}r(\sigma^{ij}-\frac{D^iuD^ju}{v^2}),$$
where
\begin{equation}\label{2.7}
v^2=1+\sigma^{ij}D_iuD_ju=1+\mid Du\mid^2.
\end{equation}
Let $h_{ij}$ be the second fundamental form of $M\subset
\mathbb{H}^{n+1}$ in term of the tangential vector fields $\{e_i,
..., e_n\}$. Then,
$$h_{ij}=\frac{\sinh r}{v}\bigg(\cosh r(D_iuD_ju+\sigma_{ij})-D_{i}D_{j}u\bigg)$$
and
\begin{equation}\label{2.07}
h^{i}_{j}=\frac{1}{v \sinh r }(\cosh
r\delta^{i}_{j}-\widetilde{g}^{ik}D_jD_k u),
\end{equation}
where $\widetilde{g}^{ij}=\sigma^{ij}-\frac{D^i u D^j u}{v^2}$.
\end{lemma}

Consider the function
\begin{eqnarray*}
\Lambda(r)=\int_{0}^{r}\sinh s d s
\end{eqnarray*}
and the vector field
\begin{eqnarray*}
V=\sinh r\partial_r,
\end{eqnarray*}
which is a conformal killing field in $\mathbb{H}^{n+1}$. Then, we
need the following lemma for $\Lambda$ and the support function
$\langle V, \nu\rangle$ of the hypersurface $M\subset
\mathbb{H}^{n+1}$.
\begin{lemma}\label{supp}
\begin{eqnarray*}
\nabla_i\nabla_j \Lambda=coshr g_{ij}-h_{ij}\langle V, \nu\rangle
\end{eqnarray*}
and
\begin{eqnarray*}
\langle V, \nu\rangle_{ij}=\cosh rh_{ij}+\sinh rg^{pq}h_{ij;
p}\nabla_{q} r-h_{im}h^{m}_{j}\langle V, \nu\rangle.
\end{eqnarray*}
\end{lemma}

See Lemma 2.2 and Lemma 2.6 in \cite{Guan15} or \cite{Jin} for the
proof.

\section{The a prior estimates}

In order to prove Theorem \ref{Main}, we use the degree theory for
nonlinear elliptic equation developed in \cite{Li89} and the proof
here is similar to \cite{Li-Ol, Jin, An, Li-Sh}. First, we consider
the family of equations for $0\leq t\leq 1$
\begin{eqnarray}\label{Eq2}
det^{\frac{1}{n}}(\widetilde{h}_{ij})=tf(x, r)+(1-t)\phi(r)(\coth
r-1)
\end{eqnarray}
and $\phi$ is a positive function which satisfies the following
conditions:

(a) $\phi(r)>0$;

(b) $\phi(r)>1$ for $r\leq r_1$;

(c) $\phi(r)<1$ for $r\geq r_2$;

(d) $\phi^{\prime}(r)<0$.

\subsection{$C^0$ Estimates}

Now, we can prove the following proposition which asserts that  the
solution of the equation \eqref{Eq} have uniform $C^0$ bound.

\begin{proposition}\label{C^0}
Under the assumptions \eqref{ASS1} and \eqref{ASS2} mentioned in
Theorem \ref{Main}, if the horo-convex hypersurface $M=\{(x, r(x)):
x \in \mathbb{S}^n\}\subset \mathbb{H}^{n+1}$ satisfies the equation
\eqref{Eq2} for a given $t \in [0, 1]$, then
\begin{eqnarray*}
r_1<r(x)<r_2, \quad \forall \ x \in \mathbb{S}^n.
\end{eqnarray*}
\end{proposition}

\begin{proof}
Assume $r(x)$ attains its maximum at $x_0 \in \mathbb{S}^n$ and
$r(x_0)\geq r_2$, then recalling \eqref{2.07}
\begin{eqnarray*}
h^{i}_{j}=\frac{1}{v \sinh r }(\cosh
r\delta^{i}_{j}-\widetilde{g}^{ik}D_k D_ju),
\end{eqnarray*}
which implies together with the fact the matrix $D_iD_ju$ is
non-positive definite at $x_0$
\begin{eqnarray*}
h^{i}_{j}(x_0)-\delta^{i}_{j}=\frac{1}{\sinh r }(\cosh
r\delta^{i}_{j}-D^iD_ju)-\delta^{i}_{j}\geq(\coth
r-1)\delta^{i}_{j}.
\end{eqnarray*}
Thus, we have at $x_0$
\begin{eqnarray*}
det^{\frac{1}{n}}(\widetilde{h}_{ij})\geq (\coth r-1).
\end{eqnarray*}
So, we arrive at $x_0$
\begin{eqnarray*}
t f(x, r)+(1-t)\phi(r)(\coth r-1)\geq (\coth r-1).
\end{eqnarray*}
Thus, we obtain at $x_0$
\begin{eqnarray*}
f(x, r)> (\coth r-1),
\end{eqnarray*}
which is in contradiction with \eqref{ASS1}. Thus, we have $r(x)<
r_2$ for $x \in \mathbb{S}^n$. Similarly, we can obtain $r(x)> r_1$
for $x \in \mathbb{S}^n$.
\end{proof}

Now, we prove the following uniqueness result.

\begin{proposition}\label{Uni}
For $t=0$, there exists an unique horo-convex solution of the
equation \eqref{Eq2}, namely $M=\{(x, r(x)) \in \mathbb{H}^{n+1}:
r(x)=r_0\}$, where $r_0$ satisfies $\varphi(r_0)=1$.
\end{proposition}

\begin{proof}
Let $M$ be a solution of \eqref{Eq2},  for $t=0$
\begin{eqnarray*}
det^{\frac{1}{n}}(\widetilde{h}_{ij})-\phi(r)(\coth r-1)=0.
\end{eqnarray*}
Assume $r(x)$ attains its maximum $r_{max}$ at $x_0 \in
\mathbb{S}^n$, then we have at $x_0$
\begin{equation*}
h^{i}_{j}=\frac{1}{\sinh r}(\cosh r\delta^{i}_{j}-D^iD_ju),
\end{equation*}
which implies together with the fact the matrix $D_iD_ju$ is
non-positive definite at $x_0$
\begin{eqnarray*}
\det(\widetilde{h}_{ij})\geq (\coth r-1)^n.
\end{eqnarray*}
Thus, we have by the equation \eqref{Eq2}
\begin{eqnarray*}
\varphi(r_{max})\geq 1.
\end{eqnarray*}
Similarly,
\begin{eqnarray*}
\varphi(r_{min})\leq 1.
\end{eqnarray*}
Thus, since $\varphi$ is a decreasing function, we obtain
\begin{eqnarray*}
\varphi(r_{min})=\varphi(r_{max})=1.
\end{eqnarray*}
We conclude
\begin{eqnarray*}
r(x)=r_0
\end{eqnarray*}
for any  $(x, r(x)) \in M$, where $r_0$ is the unique solution of
$\varphi(r_0)=1$.
\end{proof}

\subsection{$C^1$ Estimates}

For the prescribed Weingarten curvature problem, it is need to
impose a sign condition for radial derivative of $f$ to derived the
a prior gradient estimate. However, such condition is not need here,
since the horo-convexity of the hypersurface in $\mathbb{H}^{n+1}$
automatically yields the bounded of the gradient of the radial
function.

\begin{proposition}\label{C^1}
If the horo-convex hypersurface $M=\{(x, r(x)): x \in
\mathbb{S}^n\}\subset \mathbb{H}^{n+1}$ satisfies \eqref{Eq2}, then
there exists a constant $C$ depending on the minimum and maximum
values of $r$ such that
\begin{eqnarray*}
|D r(x)|\leq C, \quad  \forall~ x \in \mathbb{S}^n.
\end{eqnarray*}
\end{proposition}

\begin{proof}
Assume $x_0$ is the maximum value point of $|Du|^2(x)$. Thus, we
arrive at $x_0$,
\begin{eqnarray*}
\sigma^{kl}D_lu D_{k}D_ju=0
\end{eqnarray*}
for all $1\leq j\leq n$, which implies at $x_0$
\begin{eqnarray*}
D_iuh^{i}_{k}g^{jk}D_ju&=&\frac{D_iu}{v \sinh r }(\cosh
r\delta^{i}_{k}-\widetilde{g}^{il}D_l D_ku)\sinh^{-2}
r\frac{D^ku}{v^2}\\&=&\frac{1}{v^3 }\frac{\coth r }{\sinh^{3}
r}|Du|^2.
\end{eqnarray*}
in view of
\begin{equation*}
g^{jk}D_ju=\sinh^{-2} r\Big(\sigma^{jk}-\frac{D^ju
D^ku}{v^2}\Big)D_ju=\sinh^{-2} r\frac{D^ku}{v^2}.
\end{equation*}
Since $M$ is horo-convex, so $h^{i}_{j}>\delta^{i}_{j}$. Thus, we
obtain at $x_0$
\begin{equation*}
\frac{1}{v^3 }\frac{\coth r }{\sinh^{3} r}|Du|^2> \sinh^{-2}
r\frac{|Du|^2}{v^2},
\end{equation*}
which implies
\begin{equation*}
\frac{\coth r}{\sinh r}>\sqrt{1+|Du|^2}.
\end{equation*}
So, our proof is completed.
\end{proof}

\subsection{$C^2$ Estimates}

For convenience, we denote by
\begin{eqnarray*}
G(\widetilde{h}_{ij})\allowdisplaybreaks\notag=det^{\frac{1}{n}}(\widetilde{h}_{ij}),
\quad  \widetilde{f}(x, r)=t f(x, r)+(1-t)\phi(r)(\coth r-1),
\end{eqnarray*}
and
\begin{eqnarray*}
G^{ij}(\widetilde{h}_{ij})=\frac{\partial G}{\partial
\widetilde{h}_{ij}}, \quad G^{ij, r
s}(\widetilde{h}_{ij})=\frac{\partial^2 G}{\partial
\widetilde{h}_{ij}\partial\widetilde{h}_{rs}}.
\end{eqnarray*}

To estimate the second fundamental form of $M$, we need the
following two lemmas.
\begin{lemma}\label{C^2-2}
Let $M=\{(x, r(x)): x \in \mathbb{S}^n\}\subset \mathbb{H}^{n+1}$ be
a horo-convex solution of \eqref{Eq2}, then we have the following
equality
\begin{eqnarray*}
&&G^{ij}\langle V, \nu\rangle_{ij}+\langle V, \nu\rangle
G^{ij}h_{im}h^{m}_{j}=\sinh r\nabla_p\widetilde{ f}\nabla^{p}
r+\cosh r(\widetilde{f}+G^{ij}g_{ij}).
\end{eqnarray*}
\end{lemma}

\begin{proof}
We have by Lemma \ref{supp}
\begin{eqnarray*}
\langle V, \nu\rangle_{ij}=\cosh rh_{ij}+\sinh rg^{pq}h_{ij;
p}\nabla_{q} r-h_{im}h^{m}_{j}\langle V, \nu\rangle,
\end{eqnarray*}
which results in
\begin{eqnarray*}
G^{ij}\langle V, \nu\rangle_{ij}=\sinh rG^{ij}h_{ij; p}\nabla^{p}
r+\cosh r G^{ij}h_{ij}-\langle V,\nu\rangle G^{ij}h_{im}h^{m}_{j}.
\end{eqnarray*}
Differentiating the equation \eqref{Eq2} once, we have
$$G^{ij}h_{ij; p}=\nabla_p \widetilde{f},$$
which implies together with $G^{ij}h_{ij}=G+G^{ij}g_{ij}$
\begin{eqnarray*}
&&G^{ij}\langle V, \nu\rangle_{ij}+\langle V, \nu\rangle
G^{ij}h_{im}h^{m}_{j}=\sinh r\nabla_p \widetilde{f}\nabla^{p}
r+\cosh r(G+G^{ij}g_{ij}).
\end{eqnarray*}
Therefore we complete the proof.
\end{proof}

\begin{lemma}\label{f-2}
Let $M=\{(x, r(x)): x \in \mathbb{S}^n\}\subset \mathbb{H}^{n+1}$ be
a horo-convex solution of \eqref{Eq2}, then we have the following
equalities for $\widetilde{f}$
\begin{eqnarray*}
|\nabla\widetilde{f}|\leq C|\widetilde{f}|_{C^1}(1+|D r|)
\end{eqnarray*}
and
\begin{eqnarray*}
|\nabla^2\widetilde{f}|\leq C|\widetilde{f}|_{C^2}(1+|Dr|^2+H),
\end{eqnarray*}
where the constant $C$ depends on the minimum and maximum values of
$r$.
\end{lemma}

\begin{proof}
A direct calculation implies
\begin{eqnarray*}
\nabla_p\widetilde{f}=\widetilde{f}_{p}+\widetilde{f}_{r}r_p.
\end{eqnarray*}
Thus, we have by noticing that
$g^{ij}=\sinh^{-2}r(\sigma^{ij}-\frac{D^iu D^j u}{v^2})$
\begin{eqnarray*}
|\nabla\widetilde{f}|^2=g^{pq}\widetilde{f}_{p}\widetilde{f}_{q}+\widetilde{f}^{2}_{r}r_pr_qg^{pq}\leq
C|\widetilde{f}|^{2}_{C^1}(1+|D r|^2).
\end{eqnarray*}
Moreover, we have
\begin{eqnarray*}
\nabla_p\nabla_p\widetilde{f}=\widetilde{f}_{pp}+2\widetilde{f}_{pr}r_p+\widetilde{f}_{rr}(r_p)^2
+\widetilde{f}_{r}r_{pp},
\end{eqnarray*}
and we know from Lemma \ref{supp}
\begin{eqnarray*}
\sinh r\nabla_i\nabla_j r+\cosh r\nabla_i r \nabla_j r=\cosh
rg_{ij}-h_{ij}\langle V, \nu\rangle.
\end{eqnarray*}
Thus, we arrive
\begin{eqnarray*}
|\nabla^2 r|\leq C|f|_{C^2}(1+|Dr|^2+H).
\end{eqnarray*}
\end{proof}

Now we begin to estimate the second fundamental form.

\begin{proposition}\label{C^2}
If the horo-convex hypersurface $M=\{(x, r(x)): x \in
\mathbb{S}^n\}\subset \mathbb{H}^{n+1}$ satisfies \eqref{Eq2} and
$f(x, r)\in C^\infty(\mathbb{S}^{n}\times [0, +\infty))$ is a
positive function, then there exists a constant $C$, depending on
$n$, $|f|_{C^2}$ and $|r|_{C^1(\mathbb{S}^n)}$ such that
\begin{equation*}
|\kappa_{i}(p)|\le C, \quad \forall ~ p \in M, 1\leq i\leq n,
\end{equation*}
where $\kappa_i$ is the principal curvature of $M$.
\end{proposition}

\begin{proof}
Since $M$ is horo-convex, we only need to estimate the mean
curvature $H$ of $M$. Taking the auxillary function
$$W(X)=\log H-\log\langle V, \nu\rangle.$$ Assume that $p_0$ is the
maximum point of $W$. Then at $p_0,$
\begin{equation}\label{C2-1}
0=W_{i}=\frac{H_i}{H}-\frac{\langle V, \nu\rangle_i}{\langle V,
\nu\rangle}
\end{equation}
and
\begin{equation}\label{C2-2}
0\geq W_{ij}=\frac{H_{ij}}{H}-\frac{\langle V,
\nu\rangle_{ij}}{\langle V, \nu\rangle}.
\end{equation}
Choosing a suitable coordinate $\{x^1, x^2, ..., x^n\}$ on the
neighborhood of $p_0 \in M$ such that the matrix
$g_{ij}(p_0)=\delta_{ij}$ and $\{h_{ij}\}$ is diagonal at $p_0$.
This implies at $p_0$
\begin{equation*}
0\geq G^{ij}W_{ij}=\sum_{l=1}^{n}\frac{1}{H}G^{ii}h_{ll;
ii}-\frac{G^{ii}\langle V, \nu\rangle_{ii}}{\langle V, \nu\rangle}.
\end{equation*}
From \eqref{2rd}, we obtain
\begin{eqnarray*}
H_{ii}=\Delta h_{ii}-(h^{2}_{ii}+1)H+(|A|^2+n)h_{ii},
\end{eqnarray*}
which results in at $p_0$
\begin{eqnarray*}
0&\geq&\frac{1}{H}\sum_{l=1}^{n}G^{ii}h_{ii;
ll}-G^{ii}(h^{2}_{ii}+1)+G^{ii}h_{ii}\frac{|A|^2+n}{H}-\frac{G^{ii}\langle
V, \nu\rangle_{ii}}{\langle V, \nu\rangle}.
\end{eqnarray*}
Differentiating the equation \eqref{Eq2} twice, we have
\begin{eqnarray}\label{D-1}
G^{ij}h_{ij; l}=\nabla_l\widetilde{f}
\end{eqnarray}
and
\begin{eqnarray}\label{D-2}
G^{ij,rs}h_{ij; l}h_{rs; l} +G^{ij}h_{ij;
ll}=\nabla_l\nabla_l\widetilde{f}.
\end{eqnarray}
We know from Lemma \ref{f-2}
\begin{eqnarray}\label{D-3}
|\nabla_l\widetilde{f}|\leq C, \quad
|\nabla_l\nabla_l\widetilde{f}|\leq C(1+H).
\end{eqnarray}
Combining \eqref{D-1}, \eqref{D-2} and \eqref{D-3}, we arrive at
$p_0$ by Lemma \ref{C^2-2} and the inequality $|A|^2\geq
\frac{H^2}{n}$
\begin{eqnarray*}
0&\geq&\frac{1}{H}\sum_{l=1}^{n}\nabla_l\nabla_l\widetilde{f}-G^{ii}(h^{2}_{ii}+1)+G^{ii}h_{ii}\frac{|A|^2+n}{H}-\frac{G^{ii}\langle
V, \nu\rangle_{ii}}{\langle V, \nu\rangle}\\&\geq&
-C\frac{H+1}{H}-G^{ij}g_{ij}+G^{ij}h_{ij}\frac{|A|^2+n}{H}\\&&-\frac{\sinh
r}{\langle V, \nu\rangle}\nabla_l \widetilde{f}\nabla^{l}
r-\frac{\cosh r}{\langle V, \nu\rangle}G^{ij}h_{ij}\\&\geq&
-C\frac{H+1}{H}-G^{ij}g_{ij}+(\widetilde{f}+G^{ij}g_{ij})\frac{|H|^2+n^2}{n
H}\\&&-C-CG^{ij}g_{ij}\\&\geq& C H-C,
\end{eqnarray*}
if we assume $H$ is big enough, otherwise our proposition holds
true. So we can derive $H\le C$ at $p_0$. So, our proof is
completed.
\end{proof}

\section{The proof of  Theorem  \ref{Main}}

In this section, we can use the degree theory for nonlinear elliptic
equation developed in \cite{Li89} to prove Theorem \ref{Main}. We
only sketch will be given below, see \cite{Li-Ol, An, Jin, Li-Sh}
for details.

After establishing the  a priori estimates Proposition \ref{C^0},
Proposition \ref{C^1} and Proposition \ref{C^2}, we know that the
equation \eqref{Eq2} is uniformly elliptic. From \cite{Eva82,
Kry83}, and Schauder estimates, we have
\begin{eqnarray}\label{C2+}
|r|_{C^{4,\alpha}(\mathbb{S}^n)}\leq C
\end{eqnarray}
for any horo-convex solution $M$ to the equation \eqref{Eq}. We
define
\begin{eqnarray*}
C_{0}^{4,\alpha}(\mathbb{S}^n)=\bigg\{r \in
C^{4,\alpha}(\mathbb{S}^n): M=\{(x, r(x)): x\in \mathbb{S}^n\} \
\mbox{is}
 \ \mbox{horo-convex}\bigg\}.
\end{eqnarray*}
Let us consider $$F(\cdot, t):
C_{0}^{4,\alpha}(\mathbb{S}^n)\rightarrow
C^{2,\alpha}(\mathbb{S}^n),$$ which is defined by
\begin{eqnarray*}
F(r, t)=det^{\frac{1}{n}}(\widetilde{h}_{ij})-t f(x,
r)-(1-t)\phi(r)(\coth r-1).
\end{eqnarray*}
Let $$\mathcal{O}_R=\{r \in C_{0}^{4,\alpha}(\mathbb{S}^n):
|r|_{C^{4,\alpha}(\mathbb{S}^n)}<R\},$$ which clearly is an open set
of $C_{0}^{4,\alpha}(\mathbb{S}^n)$. Moreover, if $R$ is
sufficiently large, $F(r, t)=0$ has no solution on $\partial
\mathcal{O}_R$ by the a prior estimate established in \eqref{C2+}.
Therefore the degree $\deg(F(\cdot, t), \mathcal{O}_R, 0)$ is
well-defined for $0\leq t\leq 1$. Using the homotopic invariance of
the degree, we have
\begin{eqnarray*}
\deg(F(\cdot, 1), \mathcal{O}_R, 0)=\deg(F(\cdot, 0), \mathcal{O}_R,
0).
\end{eqnarray*}
Proposition \ref{Uni} shows that $r=r_0$ is the unique solution to
the above equation for $t=0$. Direct calculation show that
\begin{eqnarray*}
F(sr_0, 0)=[1-\varphi(sr_0)](\coth sr_0-1).
\end{eqnarray*}
Using the fact $\varphi(r_0)=1$, we have
\begin{eqnarray*}
\delta_{r_0}F(r_0, 0)=\frac{d}{d s}|_{s=1}F(sr_0,
0)=-\varphi^{\prime}(r_0)(\coth r_0-1)>0,
\end{eqnarray*}
where $\delta F(r_0, 0)$ is the linearized operator of $F$ at $r_0$.
Clearly, $\delta F(r_0, 0)$ takes the form
\begin{eqnarray*}
\delta_{w}F(r_0, 0)=-a^{ij}w_{ij}+b^i
w_i-\varphi^{\prime}(r_0)(\coth r_0-1) w,
\end{eqnarray*}
where $a^{ij}$ is a positive definite matrix. Since
$-\varphi^{\prime}(r_0)(\coth r_0-1)>0,$ thus $\delta F(r_0, 0)$ is
an invertible operator. Therefore,
\begin{eqnarray*}
\deg(F(\cdot, 1), \mathcal{O}_R, 0)=\deg(F(\cdot, 0), \mathcal{O}_R,
0)=\pm 1.
\end{eqnarray*}
So, we obtain at least a solution at $t=1$. This completes the proof
of Theorem \ref{Main}.

\end{document}